\documentclass[12pt]{amsart}
\usepackage[latin1]{inputenc}      % pour les accents
\usepackage{tikz}
\usetikzlibrary{calc,through,backgrounds}
\usetikzlibrary{snakes}
\usetikzlibrary{arrows}
\usepackage{amssymb}
\usepackage{amsmath}
\usepackage{enumitem}

\newcommand{\R}{{\mathbb R}}
\newcommand{\Z}{{\mathbb Z}}
\newcommand{\N}{{\mathbb N}}

\newtheorem{theo}{Theorem}

\newtheorem{coro}[theo]{Corollary}
\newtheorem{lemma}[theo]{Lemma}

\newenvironment{proofof}[1]{\noindent {\bf Proof of #1.}}{ \hfill\qed\\ }

\newcommand{\rset}[2] {\left\{ #1 \: \left| \: #2 \right. \! \right\} }
\newcommand{\lset}[2] {\left\{ \left. \! #1 \: \right| \: #2 \right\} }

\newcommand{\bi} {billiard}

\newcommand{\me} {measure}
\newcommand{\tr} {trajector}
\newcommand{\erg} {ergodic}
\newcommand{\sy} {system}
\newcommand{\hyp} {hyperbolic}
\newcommand{\pr} {probability}
\newcommand{\ra} {random}
\newcommand{\dsy} {dynamical system}
\renewcommand{\o} {orbit}

\newtheorem{rema}[theo]{Remark}
\newcommand{\sca} {scatterer}
\newcommand{\con} {configuration}
\newcommand{\bo} {\gamma}
\newcommand{\cs} {\Gamma}
\newcommand{\ta} {\mathcal{Q}}
\renewcommand{\l} {\ell}
\newcommand{\bn} {\mathbf{n}}
\newcommand{\ps} {\mathcal{M}}

\begin{document}

\title{Infinite-horizon Lorentz tubes and gases: recurrence and 
ergodic properties}

\author{Marco Lenci}
\address{Dipartimento di Matematica, Universit\`a di Bologna\\
Piazza di Porta San Donato 5, 40126 Bologna, Italy}
\email{marco.lenci@unibo.it}
\urladdr{http://www.dm.unibo.it/{\lower.7ex\hbox{\~{}}}lenci/} \date{}

\author{Serge Troubetzkoy}
\address{Centre de physique th\'eorique\\
Federation de Recherches des Unites de Mathematique de Marseille\\
Institut de math\'ematiques de Luminy and\\ 
Universit\'e de la M\'editerran\'ee\\ 
Luminy, Case 907, F-13288 Marseille Cedex 9, France}
\email{troubetz@iml.univ-mrs.fr}
\urladdr{http://iml.univ-mrs.fr/{\lower.7ex\hbox{\~{}}}troubetz/} \date{}

\begin{abstract}
  We construct classes of two-dimensional aperiodic \linebreak Lorentz
  systems that have infinite horizon and are `chaotic', in the sense
  that they are (Poincar\'e) recurrent, uniformly hyperbolic and
  ergodic, and the first-return map to any scatterer is $K$-mixing. In
  the case of the Lorentz tubes (i.e., Lorentz gases in a strip), we
  define general measured families of systems (\emph{ensembles}) for
  which the above properties occur with probability 1. In the case of
  the Lorentz gases in the plane, we define families, endowed with a
  natural metric, within which the set of all chaotic dynamical
  systems is uncountable and dense.
  \vspace{5pt} 
\end{abstract}

\date{\vspace*{3pt}
July 12, 2011}

\maketitle 
\markboth{Marco Lenci, Serge Troubetzkoy}
{Lorentz tubes and gases}

\section{Introduction}
\label{sec-intro}

A Lorentz \sy\ is a \dsy\ of a point particle moving inertially in an
unbounded domain (in this paper we consider only planar domains)
endowed with an infinite number of dispersing (i.e., locally convex)
\sca s. When the particle hits a \sca, which is regarded as infinitely
massive, it undergoes an elastic collision: the angle of reflection
equals the angle of incidence.

The most popular such \sy\ is undoubtedly the Lorentz gas, devised by
Lorentz in 1905 \cite{Lo} to study the dynamics of an electron in a
crystal; the term `gas' was introduced later in the century, when
versions of the Lorentz model were used to give a statistical
description of the motion of a molecule in a gas.%
\footnote{To our knowledge, the first appearance of this model within
  the scope of the kinetic theory of gases is in a 1932 textbook of
  theoretical physics \cite{J}. The model is used briefly for the
  estimation of the number of collisions per unit time of a molecule
  in a gas; however, no mention of Lorentz is made. The first
  occurrence of the phrase `Lorentz gas' in the scientific literature
  seems to date back to 1941 \cite{Gr}. Curiously, the first occurrence
  of the phrase `Lorentzian gas' that we are aware of is found in the
  article preceding \cite{Gr} in the same issue of the same journal
  \cite{HIV}. Finally, it is interesting to notice that, although
  Lorentz's original papers \cite{Lo} are about electrons in a metal
  and not molecules in a gas, he treats the problem by deriving and
  solving a Boltzmann-like transport equation; see also \cite{K}.}

In the past century, Lorentz \sy s have been preferred models in the
fields of statistical physics, optics, acoustics, and generally
anywhere the diffusive properties of a chaotic motion were to be
investigated.  Throughout this history, for reasons of mathematical
convenience, the models that were studied most often and most deeply
were \emph{periodic} (i.e., the \con\ of \sca s was invariant for the
action of a discrete group of translations) and with \emph{finite
horizon} (i.e., the free flight was bounded above). Only recently
have aperiodic Lorentz \sy s come to the fore \cite{Le1, Le2, DSV,
CLS, SLDC, Tr2} (within the scope of \dsy s, that is, aperiodic
models had already been studied in other contexts; e.g., \cite{Ga,
BBS, BBP}).  However, while there is a considerable literature on
infinite-horizon periodic Lorentz gases, almost nothing is known, at
least to these authors, on aperiodic Lorentz \sy s with infinite
horizon.

In this note we consider 2D Lorentz gases and also Lorentz tubes, that
is, Lorentz \sy s confined to a strip of $\R^2$ \cite{CLS, SLDC}.  We
construct \bi s that have infinite horizon and possess the \erg\
properties that one would expect of these chaotic \sy s, such as \erg
ity and strong mixing properties. For these infinite-\me\ preserving
\dsy s, it turns out that (Poincar\'e) recurrence is not only a
necessary but also a sufficient condition for \erg ity; this is a
consequence of the \hyp\ structure that our \sy s can be shown to have
\cite{Le1, CLS}.  As for mixing, since a universally accepted
definition of mixing is not available in infinite \erg\ theory (see,
e.g., the discussion in \cite{Le3}), we characterize this aspect of
the dynamics by proving that certain first-return maps are
$K$-automorphisms (which implies strong mixing). This is again a
consequence of recurrence and \hyp ity.

Another important question that we aim to discuss is that of the
typicality of such \erg\ properties. One would expect that, when the
effective dimension is one (Lorentz tubes) or two (Lorentz gases),
recurrence, and all the stochastic properties that it entails, hold
for ``most'' \sy s (and one would expect the former case to be more
easily worked out than the latter).

We address this question by introducing the following two classes of
\bi s. A strip and the plane are partitioned into infinitely
many congruent \emph{cells}.  In each cell we place a \con\ of
\emph{dispersing} \sca s, i.e., a finite union of piecewise smooth
closed sets, whose smooth boundary components are seen as convex from
the exterior. The set of all the \emph{global} \con s of \sca s gives
rise to a family of \dsy s of the same type. A natural distance
between two \con s can be defined, which makes the above set a metric
space.  Furthermore, if the \con\ is chosen according to a \pr\ law,
the family becomes a \me d family, or an \emph{ensemble}, of \dsy
s. This structure is often referred to as a \emph{quenched random
  \dsy}.

In the case of the tubes we give fairly explicit sufficient conditions
for the above-mentioned \erg\ properties, thus proving that, for many
reasonable \ra\ laws on the ``disorder'' (including all non-degenerate
Bernoulli \me s), such properties hold almost surely in the ensemble;
we state these results in Section \ref{sec-tubes}. For the harder case
of the gases, we prove that the set of the \erg\ \sy s is uncountable
and dense within the whole space; this is described in Section
\ref{sec-gases}. Outlines of the proofs are given in Section
\ref{sec-proofs}.

It is worthwhile to mention that the \bi s we construct have infinite
but \emph{locally finite} horizon, that is, though the free flight has
no upper bound, no straight line exists that intersects no \sca s.

\medskip\noindent 
\textbf{Acknowledgments.} We thank A.~J.~Kox and J.~Lebowitz for
helping us with the history of the Lorentz gas, and two anonymous
referees for their careful reading of the first version of this paper.
M.~L.~is partially supported by the FIRB-``Futuro in Ricerca'' Project
RBFR08UH60 (MIUR, Italy). S.~T.~is partially supported by Projet ANR
``Perturbations'' (France).

\section{Lorentz tubes}
\label{sec-tubes}

Let $C$ denote the unit square, which will be henceforth referred to
as the \emph{cell}. Let $G^1$ denote an open segment along one of the
sides of $C$, say the left one, and $G^2$ the corresponding segment on
the opposite side (via the natural orthogonal projection). $G^1$ and
$G^2$ are called the \emph{gates} of $C$.

A \emph{local \con\ of \sca s} is a ``fat'' closed subset $\cs \subset
C$ (this means that $\cs$ is the closure of its interior) such that:
\begin{itemize}
\item[(A1)] $\partial \cs$ is made up of a finite number of
  $C^3$-smooth curves $\bo_i$, which may only intersect at their
  endpoints ($\bo_i$ is always considered a closed set).
\item[(A2)] $\partial C \setminus (G^1 \cup G^2) \subset \partial
  \cs$; and $G^1, G^2$ do not intersect $\partial \cs$.
\item[(A3)] Either $\bo_i$ is part of $\partial C$ or its curvature is
  bounded below by a positive constant (with the convention that
  positive curvature means that $\bo_i$ bends towards the inside of
  $\cs$).
\item[(A4)] The angle formed by $\bo_i$ and each intersecting $\bo_j$
  (or $G^j$) is non-zero.
\end{itemize}

\begin{rema}
  Observe that \emph{(A1)-(A3)} imply that the flat parts of $\partial
  \cs$ are `external' boundaries, in the sense that they can never be
  reached by a particle (the contrary would violate the hypothesis on
  the ``fatness'' of $\cs$). So we are dealing with dispersing, not
  semi-dispersing, \bi s.
\end{rema}

We consider a finite number $\cs^1, \cs^2, \ldots, \cs^m$ of such
local \con s (see Figs.~\ref{fig-t-b} and \ref{fig-t-nb}).

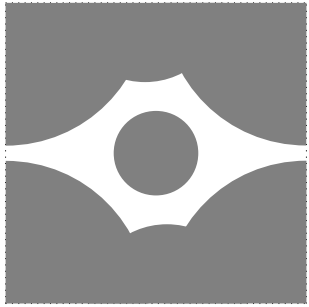
\begin{figure}[th]

\centering

\begin{tikzpicture}[scale=2]

\foreach \i in {-1,1}
\draw [dotted] (-1,\i) -- ( 1,\i);

\foreach \j in {-1,1}
\draw [dotted] (\j,-1) -- (\j,1);

\fill[color=gray] (-1,1) -- (-1,0.05) arc (-90:0:27pt) -- cycle;    
\fill[color=gray] (1,1) -- (0.05,1) arc (180:270:27pt) -- cycle;    
\fill[color=gray] (-0.6,1) -- (-0.6,1) arc (180:360:15pt) -- cycle;

\fill[color=gray] (-1,-1) -- (-0.05,-1) arc (0:90:27pt) -- cycle;    
\fill[color=gray] (1,-1) -- (1,-0.05) arc (90:180:27pt) -- cycle;    
\fill[color=gray] (0.6,-1) -- (0.6,-1) arc (0:180: 15pt) --cycle;

\fill[color=gray] (0,0) circle (8pt);  
  
\end{tikzpicture}

    \caption{A blocking cell for the Lorentz tube}
    \protect\label{fig-t-b}

\end{figure}

\begin{figure}[th]

\centering

\begin{tikzpicture}[scale=2]

\foreach \i in {-1,1}
\draw [dotted] (-1,\i) -- ( 1,\i);

\foreach \j in {-1,1}
\draw [dotted] (\j,-1) -- (\j,1);

\fill[color=gray] (-0.6,1) -- (-0.6,1) arc (180:360:15pt) 
-- (-1,1) -- (-1,0.05) arc (-90:-32:27pt) -- cycle;    
\fill[color=gray] (1,1) -- (0.05,1) arc (180:270:27pt) -- cycle;    

\fill[color=gray] (0.6,-1) -- (0.6,-1) arc (0:180: 15pt) --cycle;
\fill[color=gray] (-1,-1) -- (-0.05,-1) arc (0:90:27pt) -- cycle;    
\fill[color=gray] (1,-1) -- (1,-0.05) arc (90:180:27pt) -- cycle;    

\end{tikzpicture}

    \caption{A non-blocking cell for the Lorentz tube}
    \protect\label{fig-t-nb}

\end{figure}
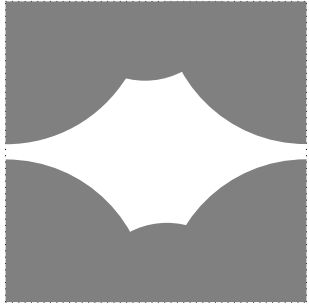

A Lorentz tube (LT) is a chain of cells $C_n$ ($n \in \Z$) such that
$G_n^2$, the right gate of $C_n$, coincides with $G_{n+1}^1$, the left
gate of $C_{n+1}$. More precisely, call $C_n := [n, n+1] \times [0,1]$
the particular copy of $C$ immersed in $\R^2$ as indicated, and denote
$\Omega := \{1, 2, \ldots, m \}$.  Then, for $\l := ( \l_n )_{n \in
\Z} \in \Omega^\Z$, we define the \emph{\bi\ table}
\begin{equation}
  \ta = \ta_\l := \bigcup_{n \in \Z} C_n \setminus \cs_n^{\l_n},
\end{equation}
where $\cs_n^{\l_n}$ is the \con\ $\cs^{\l_n}$ translated to the cell
$C_n$ (see Fig.~\ref{fig-tube}). The collection $( \cs_n^{\l_n} )_{n
\in \Z}$---equivalently $\l$---is called the \emph{global \con\ of
\sca s}.

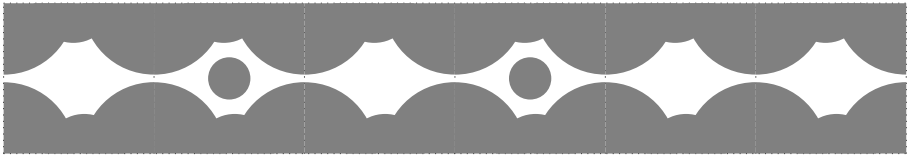
\begin{figure}[t]
\centering

\begin{tikzpicture}[scale=1]

\foreach \i in {-1,1}
\draw [dotted] (-7,\i) -- (5,\i);

\foreach \j in {-7,-5,-3,-1,1,3,5}
\draw [dotted] (\j,-1) -- (\j,1);

\foreach \j in {-7,-5,-3,-1,1,3}
\fill[color=gray] (\j,1) -- (\j,0.05) arc (-90:0:27pt) -- cycle;
\foreach \j in {-7,-5,-3,-1,1,3}
\fill[color=gray] (\j +  2,1) -- ( \j + 1.05,1) arc (180:270:27pt) -- cycle;
\foreach \j in {-7,-5,-3,-1,1,3}
\fill[color=gray] (\j+0.4,1) -- (\j+0.4,1) arc (180:360:15pt) -- cycle;

\foreach \j in {-7,-5,-3,-1,1,3}
\fill[color=gray] (\j,-1) -- (\j+1-0.05,-1) arc (0:90:27pt) -- cycle;
\foreach \j in {-7,-5,-3,-1,1,3}
\fill[color=gray] (\j+2,-1) -- (\j+2,-0.05) arc (90:180:27pt) -- cycle;
\foreach \j in {-7,-5,-3,-1,1,3}
\fill[color=gray] (\j+1.6,-1) -- (\j+1.6,-1) arc (0:180: 15pt) --cycle;

\fill[color=gray] (0,0) circle (8pt);
\fill[color=gray] (-4,0) circle (8pt);

\end{tikzpicture}

   \caption{A Lorentz tube}
   \protect\label{fig-tube}

\end{figure}

We henceforth say `the LT $\l$' to mean both the table $\ta_\l$ and
the \bi\ dynamics defined on it. By this we mean, precisely, the \dsy\
$(\ps_\l, T_\l, \mu_\l)$---more concisely, $(\ps, T, \mu)$---where:
\begin{itemize}
\item $\ps$ is the collection of all the \emph{line elements} of the
  dynamics, i.e., all the position-velocity pairs $(q,v)$, where $q
  \in \partial \ta$ and $v$ is a unit vector based in $q$ and pointing
  toward the interior of $\ta$. $(q,v)$ is meant to represent the
  dynamical variables of the particle right after a collision ($v$ can
  be chosen unitary because in this Hamiltonian \sy\ the conservation
  of energy equals the conservation of speed).
\item $T$ is the map that takes $(q,v)$ into the next post-collisional
  line element $(q',v')$, along the \bi\ \tr y of $(q,v)$; this map
  fails to be well defined only at a negligible set of line elements,
  which are called \emph{singular}, cf.\ below. $T$ is usually called
  the \emph{(standard) \bi\ map}.
\item $\mu$ is the invariant \me\ induced on the Poincar\'e section
  $\ps$ by the Liouville \me; it is well known that $d\mu(q,v) =
  \langle n_q, v \rangle dq dv$, where $n_q$ is inner unit normal to
  $\partial \ta$ in $q$.  It is easy to verify that the set of all
  singular points in phase space is null w.r.t.~$\mu$ Finally, as is
  evident, $\mu( \ps ) < \infty$ if and only if the total length of
  $\partial \ta$ is finite.
\end{itemize}
Notice that, by the definition of $\ta$, a \tr y that intersects a
gate $G_n^i$ crosses it. (The \tr y that runs along the segment
$G_n^i$ will always be considered singular.)

We assume that there are two types of local \con s: the
\emph{blocking} \con s, corresponding to the index set $\Omega_B := \{
1, 2, \ldots, m' \}$ ($m' < m$), and the \emph{non-blocking}
confugurations, corresponding to the set $\Omega_{NB} := \{ m' + 1, m'
+ 2, \ldots, m \}$. The former type verifies the following condition:
\begin{itemize}
\item[(A5)] If $a \in \Omega_B$, any \bi\ \tr y that enters a cell
  with \con\ $\cs^a$ must experience a collision before leaving it
  (Fig.~\ref{fig-t-b}).
\end{itemize}

An example of a \con\ $\cs^a$, with $a \in \Omega_{NB}$, is shown in
Fig.~\ref{fig-t-nb}.

Clearly, an LT $\l \in \Omega_B^\Z$ has finite horizon, i.e., the free
flight between two successive collisions has an upper bound. An
arbitrary LT in $\Omega^\Z$ might not have this property.

A word $a_1 a_2 \cdots a_l$, with $a_i \in \Omega$, is called a
\emph{factor} of $\l \in \Omega^\Z$ if there exists $n$ such that
$\l_{n+i} = a_i$, for $i = 1, 2, \dots, l$. The factor is called
blocking (respectively, non-blocking) if all the $a_i$ belong to
$\Omega_B$ (respectively, $\Omega_{NB}$); it is called constant if
they are all equal. The positive integer $l$ is called the length of
the factor.

The technique used by Troubetzkoy in \cite{Tr1}, when applied to the
present models, easily implies that any LT in $\Omega^\Z$ which has
arbitrarily long blocking constant factors (both forwards and
backwards) is recurrent. Furthermore, Cristadoro, Lenci and Seri have
shown that, for LTs in $\Omega_B^\Z$, \erg ity and recurrence are
equivalent \cite{CLS}, and they imply $K$-mixing for suitable return
maps (this last result is actually stated in \cite{SLDC} but its proof
applies as well to the models of \cite{CLS}).

Our results on the Lorentz tubes are a combination and an extension of
these ideas.  To describe them we introduce some notation that will
appear obscure at first, but will be explained shortly.  For fixed
$\l$, define $g_0^+ := g_0^- := 0$ and, recursively for $j > 0$,
\begin{equation}
  \label{gj-pm}
  g_{j+1}^\pm := \min \rset{k \in \Z^+} {\l_{\pm \sum_{i=0}^j g_i^\pm 
  \,\pm\, k} \in \Omega_B}.
\end{equation}
In other words, $g_j^+ - 1$ (respectively, $g_j^- - 1$) is the length
of the $j^\mathrm{th}$ non-blocking factor to the right (respectively,
to the left) of the cell $C_0$ (with the convention that between two
blocking cells there is a non-blocking factor of length 0). Notice
that, if $\l_0 \in \Omega_B$, this coding reflects exactly the
sequence of non-blocking factors of $\l$. Otherwise, there is a little
difference which, see Remark \ref{rmk1} below, does not affect the
upcoming statement.

\begin{theo} 
  \label{thm1}
  Assume \emph{(A1)-(A5)}. For any $\l \in \Omega^\Z$ which has
  arbitrarily long blocking constant factors, both forward and
  backwards, and such that both sequences $(g_j^\pm)_{j \in \N}$ grow
  at most like a power-law, the corresponding \dsy\ $(\ps, T, \mu)$
  is:
  \begin{itemize}
  \item[(a)] uniformly \hyp, in the sense that local stable and
    unstable manifolds exist at a.e.\ point of $\ps$, and the
    corresponding ($T$-invariant) laminations are absolutely
    continuous w.r.t.\ $\mu$ (see, e.g., \cite{Le1} for details); the
    contraction coefficient of $T^n$, along the stable direction, is
    bounded above by $C \lambda^n$ and its expansion coefficient,
    along the unstable direction, is bounded below by $C^{-1}
    \lambda^{-n}$, where $C>0$ and $\lambda \in (0,1)$ are uniform
    constants;
  \item[(b)] recurrent in the sense of Poincar\'e, i.e., given a
    measurable $A \subset \ps$, the \o\ of a.e.\ point in $A$ comes
    back to $A$ infinitely many times;
  \item [(c)] \erg, that is, if $T(A) = A$ mod $\mu$, then either $A$
    or its complement has \me\ zero.
  \end{itemize}
  Furthermore, the first-return map to any smooth component of
  $\partial \ta$, of the type $\bo_i$ as in \emph{(A1)}, is K-mixing.
\end{theo}

\begin{rema} 
  \label{rmk1}
  In the above theorem, both hypotheses are shift-invariant in
  $\Omega^\Z$, as a relocation of the origin on $\l$ will produce at
  most a shift in $(g_j^\pm)$ and a change of the first few
  terms. Therefore, when convenient---to do away with the problem
  mentioned in the previous paragraph---we make the convention that
  any $\l$ is shifted to the left the minimum amount of times for
  $C_0$ to have a blocking \con.
\end{rema} 

An LT as described in Theorem \ref{thm1} need not have infinite
horizon. To start with, it is necessary that the non-blocking \con s
let free \tr ies through, something that was not postulated. But more
is needed as well. For instance, if one designs the local \con s
$\cs^a$, with $a \in \Omega_{NB}$, so that, for any $l \in \Z^+$, all
non-blocking factors of length $l$ admit a free flight of length $\ge
l$ (and this is easy, cf.~Fig.~\ref{fig-t-nb}), then a necessary and
sufficient condition for an LT to have infinite horizon is that at
least one of the two sequences $(g_j^\pm)$ is unbounded. Or one might
ask for something less: for example, that just the \emph{constant}
non-blocking factors admit long free flights. This is enough to
guarantee that very many LTs have an infinite horizon, cf.\ Corollary
\ref{cor1}.

An important question is whether the LTs to which Theorem \ref{thm1}
applies are typical in $\Omega^\Z$. This of course depends on the
definition of `typical'. One strong notion of typicality is the
\me-theoretic notion, provided a probability \me\ $\Pi$ is put on
$\Omega^\Z$ (endowed with the natural $\sigma$-algebra generated by
the cylinders). In this case, $(\Omega^\Z, \Pi)$ becomes a \me d
family (in jargon, an \emph{ensemble}) of \dsy s, which we call
\emph{quenched \ra\ Lorentz tube}.

Many reasonable \me s $\Pi$ ensure that the assertions of Theorem
\ref{thm1} hold $\Pi$-almost surely. Here is an example:

\begin{coro} 
  \label{cor1} 
  Let $(p_1, p_2, \ldots, p_m)$ be a stochastic vector, with $p_a >
  0$, and let $\Pi$ be the Bernoulli \me\ on $\Omega^\Z$ relative to
  that vector (i.e., the unique \me\ for which $\Pi(\l_n = a) = p_a$,
  for all $n$). Then $\Pi$-a.e.\ LT in $\Omega^\Z$ has the properties
  stated in Theorem \ref{thm1}. Furthermore, if any non-blocking
  constant factor of length $l$ ($\forall l \in \Z^+$) admits a free
  flight of length $\ge l$, then a.e.\ LT has infinite horizon as
  well.
\end{coro}

\begin{proofof}{Corollary \ref{cor1}} 
  Let $p_B := \sum_{a=1}^{m'} p_a$ and $p_{NB} := \sum_{a=m' + 1}^{m}
  p_a$.  By hypothesis, $p_B, p_{NB} \in (0,1)$. It is evident that,
  w.r.t.~$\Pi$, the \ra\ variables $(g_j^\pm)$ are i.i.d., with
  distribution $\Pi( g_j^\pm = k) = p_B (p_{NB})^{k-1}$ $(k \in
  \Z^+)$.
  
  We claim that, for $\Pi$-a.e.~$\l$, there exists $K = K(\l)$ such
  that $g_j^\pm \le K j$. In fact, observe that
  \begin{equation}
     \Pi( g_j^\pm \ge j ) = (p_{NB})^{j - 1}.
  \end{equation}
  This implies that the probabilities of the `events' $\{ g_j^\pm \ge
  j \}$ form a summable sequence. Thus, by Borel-Cantelli, for
  a.e.~$\l$, the inequality $g_j^\pm / j \ge 1$ is verified only for a
  finite number of $j$'s. Setting $K := \max_j (g_j^\pm / j)$ proves
  the claim.
  
  To finish the proof of Corollary \ref{cor1}, observe that, for a
  non-degenerate Bernoulli \me, a.e.~$\l$ contains arbitrarily long
  blocking and non-blocking constant factors.
\end{proofof}

\section{Lorentz gases}
\label{sec-gases}

We now turn to the truly two-dimensional case. We tile the plane with
$\Z^2$ copies of the unit square $C$.  In this case $C$ is endowed
with four gates: $G^1$ and $G^2$, congruent and opposite open segments
(say, the left and the right gate, respectively); $G^3$ and $G^4$,
again congruent and opposite open segments (the lower and the upper
gates).

Apart from this, we have the same structure as in Section
\ref{sec-tubes}: a finite number of local \con s indexed by the set
$\Omega = \Omega_B \cup \Omega_{NB}$, where $\Omega_B$ denotes the
blocking and $\Omega_{NB}$ the non-blocking \con s, as in
Figs.~\ref{fig-g-b} and \ref{fig-g-nb}.  Again, these sets verify five
assumptions, (A1), (A3), (A4), (A5), and the counterpart of (A2):

\begin{figure}[t]
\centering

\begin{tikzpicture}[scale=2]

\foreach \i in {-1,1}
\draw [dotted] (-1,\i) -- ( 1,\i);

\foreach \j in {-1,1}
\draw [dotted] (\j,-1) -- (\j,1);

\fill[color=gray] (-1,1) -- (-1,0.05) arc (-90:0:27pt) -- cycle;    
\fill[color=gray] (-1,-1) -- (-0.05,-1) arc (0:90:27pt) -- cycle;    
\fill[color=gray] (1,-1) -- (1,-0.05) arc (90:180:27pt) -- cycle;    
\fill[color=gray] (1,1) -- (0.05,1) arc (180:270:27pt) -- cycle;    

\fill[color=gray] (0,0) circle (8pt);  

\end{tikzpicture}

    \caption{A blocking cell for the Lorentz gas}
    \protect\label{fig-g-b}

\end{figure}
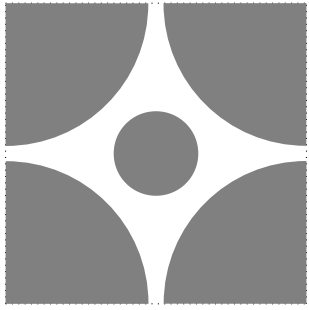

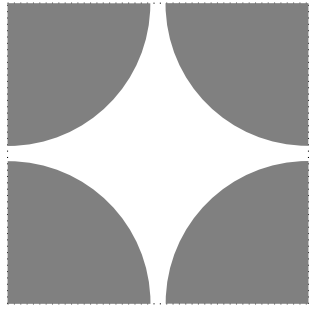
\begin{figure}[t]
\centering

\begin{tikzpicture}[scale=2]

\foreach \i in {-1,1}
\draw [dotted] (-1,\i) -- ( 1,\i);

\foreach \j in {-1,1}
\draw [dotted] (\j,-1) -- (\j,1);

\fill[color=gray] (-1,1) -- (-1,0.05) arc (-90:0:27pt) -- cycle;    
\fill[color=gray] (-1,-1) -- (-0.05,-1) arc (0:90:27pt) -- cycle;    
\fill[color=gray] (1,-1) -- (1,-0.05) arc (90:180:27pt) -- cycle;    
\fill[color=gray] (1,1) -- (0.05,1) arc (180:270:27pt) -- cycle;    

\end{tikzpicture}

    \caption{A non-blocking cell for the Lorentz gas}
    \protect\label{fig-g-nb}

\end{figure}

\begin{itemize}
\item[(A2$^\prime$)] $\partial C \setminus (G^1 \cup \ldots \cup G^4)
  \subset \partial \cs$; and $G^1, \ldots, G^4$ do not intersect
  $\partial \cs$.
\end{itemize}

A global \con\ is the collection $\l := ( \l_\bn)_{\bn \in \Z^2} \in
\Omega^{\Z^2}$, with $\bn := (n_1, n_2)$, and the \bi\ table is
\begin{equation}
  \ta = \ta_\l := \bigcup_{\bn \in \Z^2} C_\bn \setminus \cs_\bn^{\l_\bn},
\end{equation}
where $C_\bn := [n_1, n_1 + 1] \times [n_2, n_2 + 1]$, and
$\cs_\bn^{\l_\bn}$ is the \con\ $\cs^{\l_\bn}$ translated to
$C_\bn$. We call $\ta$ (and the \bi\ map thereon) a Lorentz gas
(LG)---see a realization in Fig.~\ref{fig-gas}.

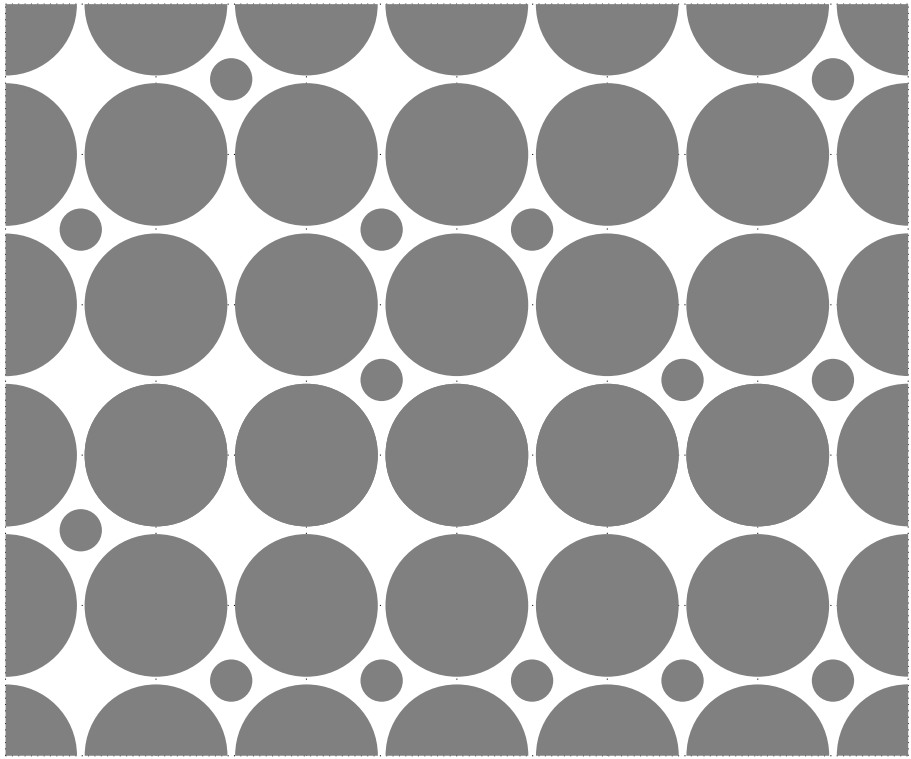
\begin{figure}[t]
\centering  

\begin{tikzpicture}[scale=1]

\foreach \i in {-3,-1,1,3,5,7}
\draw [dotted] (1,\i) -- ( 13,\i);

\foreach \j in {1,3,5,7,9,11,13}
\draw [dotted] (\j,-3) -- (\j,7);

\foreach \i in {,-1,1,3,5}
\foreach \j in {3,5,7,9,11}
\fill[color=gray] (\j,\i) circle (27pt);

\foreach \i in {1}
\foreach \j in {-2,0,2,4}
\fill[color=gray](\i,\j) -- (\i,\j+0.05) arc (-90:90:27pt) --cycle;

\foreach \i in {13}
\foreach \j in {0,2,4,6}
\fill[color=gray] (\i,\j) -- (\i,\j-0.05) arc (90:270:27pt) --cycle;

\foreach \i in {2,4,6,8,10}
\foreach \j in {7}
\fill[color=gray] (\i,\j) --  (\i+0.05,\j) arc (180:360:27pt) --cycle;

\foreach \i in {4,6,8,10,12}
\foreach \j in {-3}
\fill[color=gray](\i,\j) -- (\i-0.05,\j) arc (0:180:27pt) --cycle;

\fill[color=gray] (1,-3) -- (2-0.05,-3) arc (0:90:27pt) --cycle;
\fill[color=gray] (13,-3) -- (13,-2-0.05) arc (90:180:27pt) --cycle;
\fill[color=gray] (1,7) -- (1,6.05) arc (270:360:27pt) --cycle;
\fill[color=gray] (13,7) -- (12.05,7) arc (180:270:27pt) --cycle;

\fill[color=gray] (2,0) circle (8pt);
      
\foreach \i in {2,6,8}  \fill[color=gray] (\i,4) circle (8pt);
           
\foreach \i in {4,6,8,10,12}    
\foreach \j in {-2}
\fill[color=gray] (\i,\j) circle (8pt); 
    
\foreach \i in   {6,10,12} 
\foreach \j in {2} 
\fill[color=gray] (\i,\j) circle (8pt);
     
\foreach \i in   {4,12} 
\foreach \j in {6} 
\fill[color=gray] (\i,\j) circle (8pt);
  
\end{tikzpicture}

    \caption{A Lorentz gas}
    \protect\label{fig-gas}

\end{figure}

A word $a_1 a_2 \cdots a_l$ is called a \emph{horizontal factor} of
$\l \in \Omega^{\Z^2}$ if there exists $\bn = (n_1, n_2)$ such that
$\l_{n_1 + i, n_2} = a_i$, for $i = 1, 2, \dots, l$. The analogous
definition is given for a vertical factor. As in Section
\ref{sec-tubes}, a factor is called non-blocking if $a_i \in
\Omega_{NB}$ and constant if $a_i = a$, for all $i$.

Since we are interested in infinite-horizon \bi s, we discuss a
sufficient condition for obtaining this property. It will be seen in
Section \ref{sec-proofs} that the LGs that we construct possess
arbitrarily long, both horizontal and vertical, non-blocking constant
factors.  They have infinite horizon if all non-blocking, say,
horizontal factors of length $l$ ($\forall l \in \Z^+$) admit a free
flight of length $\ge l$.

In the two-dimensional setting no one has succeeded in proving that
recurrence is a typical property in a \me-theoretic sense (this is
actually an important open problem, cf.\ \cite{Le2, CD}). The known
results are with respect to a topological notion of typicality. For
this, we make $\Omega^{\Z^2}$ a metric space by endowing it with the
distance
\begin{equation}
  \label{dist}
  \mathrm{dist}(\l, \l') := \sum_{\bn \in \Z^2} 2^{-|n_1| -|n_2|}  \left|
  \l_\bn - \l'_\bn \right|.
\end{equation}

It was shown by Lenci that, relative to the above metric, the
Baire-typical LG in $\Omega_B^{\Z^2}$ is recurrent \cite{Le2} and
\erg\ \cite{Le1}. On the other hand, Troubetzkoy has shown that, under
the assumption that long factors of non-blocking \con s admit long
free flights, the Baire-typical LG in $\Omega^{\Z^2}$ is recurrent and
has infinite horizon \cite{Tr2}.  Once again, we combine and extend
these ideas to show that a great number of infinite-horizon LGs are
recurrent and chaotic in the sense of Theorem \ref{thm1}.

\begin{theo}
  \label{thm2}
  Assuming \emph{(A1), (A2$^\prime$), (A3)-(A5)}, the metric space
  $\Omega^{\Z^2}$ contains a dense uncountable set of LGs that are
  \hyp, recurrent and \erg\ in the sense of Theorem \ref{thm1}, and
  such that the return map to any smooth component of the type $\bo_i$
  is K-mixing. Furthermore, if, for all $l \in \Z^+$, any non-blocking
  horizontal constant factor of length $l$ admits a free flight of
  length $\ge l$, then those LGs have infinite horizon as well.
\end{theo}

\section{Proofs}
\label{sec-proofs}

\subsection{Sketch of the proof of Theorem \ref{thm1}}
\label{subs-pf1}

Let us fix an LT as in the statement of Theorem \ref{thm1}. Its
recurrence is proved essentially in the same way as for the staircase
\bi s of \cite{Tr1}. For the sake of completeness, we give an outline
of the argument.

Here and in the remainder we are going to need a notation for the
portion of the phase space pertaining to the cell $C_n$:
\begin{equation}
  \label{def-mn}
  \ps_n := \rset{(q,v) \in \ps} {q \in C_n}.
\end{equation}

It is clear by the hypotheses that the LT contains, both forward and
backwards, arbitrarily long constant factors of the \emph{same}
blocking \con, say $\cs^1$. In this paragraph and the next, when we
say blocking factor, we will always mean a factor of configurations
$\cs^1$.  Consider a blocking factor beginning with the cell $C_{n_1}$
and ending with the cell $C_{n_2}$ ($n_2 > n_1$). Let us call
$A_{n_1}^{n_2}$ the set of all the line elements of $\ps_{n_1}$ whose
\tr ies visit a cell to the right of $C_{n_2}$, before coming back to
$C_{n_1}$. Likewise, let $A_{n_2}^{n_1}$ be the set of all the line
elements of $\ps_{n_2}$ whose \tr ies visit a cell to the left of
$C_{n_1}$, before coming back to $C_{n_2}$.  Now fix $k \in
\Z^+$. Since an LT made up entirely of cells $\cs^1$ is recurrent,
there exists a positive integer $l_k$ such that, if the length of the
factor (namely, $n_2 - n_1 + 1$) is bigger than or equal to $l_k$,
both $\mu(A_{n_1}^{n_2})$ and $\mu(A_{n_2}^{n_1})$ are smaller than
$1/k$.

Consider a wandering set $W$ and, for any $n \in \Z$, set $W_n := W
\cap \ps_n$. $W_n$ is also a wandering set. For all $k \in \Z^+$,
there exist a blocking factor of length $\ge l_k$ (say, from the cell
$C_{n_1}$ to the cell $C_{n_2}$) to the right of $C_n$ and a blocking
factor of length $\ge l_k$ (say, from $C_{n_3}$ to $C_{n_4}$) to the
left of $C_n$. By definition, the \o s of the points of $W_n$ are all
disjoint and unbounded, which means they must intersect either
$A_{n_1}^{n_2}$ or $A_{n_4}^{n_3}$ in distinct points. Therefore,
using the invariance of the \me, $\mu(W_n) \le \mu( A_{n_1}^{n_2} \cup
A_{n_4}^{n_3} ) = 2/k$. Since $k$ and $n$ are arbitrary, $W$ is a null
set.

Turning to the \hyp ity and \erg ity, once the recurrence is known,
one proceeds as in \cite{Le1} or \cite{CLS}. (A fairly accurate
summary of the whole proof is given in \cite{SLDC}, although that
article refers to LTs in dimension higher than two.) Here we limit
ourselves to mentioning which parts of the proof can be worked out
using the standard techniques for classical \hyp\ \bi s \cite{KS, LW}
and which ones need to be adapted to our particular infinite-\me\ \sy.

Focusing on the \hyp ity first, it can be seen that all the arguments
used in the proof of Theorem \ref{thm1}\emph{(a)}, with one exception,
are local, that is, depend on the value of $T$ (its invariant cones,
its distortion coefficients, etc.)\ on a neighborhood of a given
point, or a given \o. In other words, they cannot distinguish between
a finite-\me\ dispersing \bi---for which everything works well
\cite{KS}---or an infinite-\me\ one.

The only argument that is not local is the one whereby $\mu$-almost
all \o s stay sufficiently far away from the singular points of $T$
(so that the singularities of the map do not interfere with the
construction of the local stable and unstable manifolds).  The
singular points are organized in smooth curves, called
\emph{singularity lines}. Each such line corresponds to all the line
elements whose first collision occurs at a given corner of $\partial
\ta$, or tangentially to a certain smooth component of it. Hence,
there are a countable number of singularity lines. They can be counted
(or at least overestimated) in a way that, in each region of the phase
space $\ps_\bo := \rset{(q,v) \in \ps} {q \in \bo}$, where $\bo$ is a
smooth portion of the boundary as in (A1), there are at most two
singularity lines for every source of singularity (a vertex or a
tangency) ``seen'' from $\bo$. (The factor 2 comes from the fact that
there are two ways to be tangent to a smooth boundary, one for each
orientation.) Also, the length of each singularity line is bounded
above by a universal constant having to do with the size of the cell.

Let us indicate with $\mathcal{S}$ the singular set of $T$, that is,
the union of all the singularity lines described above.  It turns out
that, if Lemma \ref{lm-tech} below holds, almost all \o s approach
$\mathcal{S}$ no faster than a negative power-law in time (where by
time we mean `number of collisions'). This is enough to make the
sought argument work \cite{Le1, CLS}.

Moving on to statement \emph{(c)}, as explained in \cite{Le1}, one can
exploit a suitable Local Ergodicity Theorem for \bi s (say, the
version of \cite{LW}), which uses only local arguments except in its
most delicate part, the so-called Tail Bound.  It turns out, however,
that a version of the Tail Bound can be proved for our LTs too, if the
following result holds (cf.~ Sec.~3 of \cite{SLDC}).

\begin{lemma}
  \label{lm-tech}
  There exist constants $C, \alpha > 0$ such that, for any $n \in \Z$,
  any $U \subseteq \ps_n$, and any $t \in \Z^+$, $\bigcup_{j=0}^t
  T^j(U)$ intersects at most $Ct^\alpha$ singularity lines from
  $\mathcal{S}$. (In more descriptive terms, each \tr y with initial
  conditions in $U$ may approach at most $Ct^\alpha$ singularity lines
  of $T$, in time $t$, and this bound is uniform in $U$, if $U$ is not
  too large.)
\end{lemma}

\begin{proof}
  Let us define $f_j^\pm := \pm \sum_{i=0}^j g_i^\pm$, so that $f_j^+$
  is the location of the blocking cell to the right of the
  $j^\mathrm{th}$ non-blocking factor, on the right side of the tube
  w.r.t.\ $C_0$; while $f_j^-$ is the location of the blocking cell to
  the left of the $j^\mathrm{th}$ non-blocking factor, on the left
  side of the tube; cf.\ definition (\ref{gj-pm}). It follows from the
  hypotheses of Theorem \ref{thm1} that both sequences $( |f_j^\pm|
  )_{j \in \N}$ are bounded by a power-law in $j$.
  
  Coming to the statement of the lemma, we may assume that $U = \ps_0$
  (in fact, proving the result for $U' := \ps_n$ automatically proves
  it for $U$; then, choosing $n=0$ is no loss of generality because a
  translation of the tube does not change our argument---cf.\ Remark
  \ref{rmk1}).
  
  By (A5), the configuration space \tr ies of length $t$, with initial
  conditions in $\ps_0$, cannot go further left than the cell
  $C_{f_t^-}$ or further right than $C_{f_t^+}$. The corresponding
  phase space \o s are thus contained in
  \begin{equation}
    A_t := \bigcup_{n = f_t^-}^{f_t^+} \ps_n.
  \end{equation}
  A line element in $A_t$ can ``see'' at most those cells that range
  from $C_{f_{t+1}^-}$ to $C_{f_{t+1}^+}$. Therefore, the number of
  singularity lines of $T$ in each set of the type $\ps_\bo \subset
  A_t$ ($\ps_\bo$ was defined earlier) does not exceed $C' (f_{t+1}^+
  - f_{t+1}^- + 1)$, for some $C'>0$.  But there are at most $C''
  (f_t^+ - f_t^- + 1)$ sets of that type. The product of these two
  estimates, which, as shown earlier, grows no faster than a power-law
  in $t$, is an upper bound for the number of singularity lines from
  $\mathcal{S}$ in $A_t$.
\end{proof}

Once we have local \erg ity (all but countably many points in phase
space have a neighborhood contained in one \erg\ component), global
\erg ity is easily shown.  Also, the assertion about the $K$-mixing
first-return map is proved as in \cite{SLDC}, Sec.~4.

\subsection{Sketch of the proof of Theorem \ref{thm2}}
\label{subs-pf2}

Let us endow $\Z^2$ with the norm $\| \bn \| = \| (n_1, n_2) \| :=
|n_1| + |n_2|$. For $j \in \Z^+$, set $D_j := \rset{\bn \in \Z^2} {\|
  n \| = j^2}$ (this set resembles the border of a rhombus in
$\Z^2$). Given $i \in \Z^+$, define
\begin{equation}
  \label{def-zi}
  Z_i := \Z^2 \setminus \bigcup_{j \ge i} D_j
\end{equation}
and
\begin{equation}
  \label{def-li}
  \mathcal{L}_i := \lset{ \l = ( \l_\bn) \in \Omega^{\Z^2} }
  { \l_\bn = 1, \, \forall \bn \not\in Z_i }.
\end{equation}
In other words, $\mathcal{L}_i$ is the set of all the global \con s
which have ``blocking circles'' (namely, circles filled with cells of
type $\cs^1$) at all radii $j^2$, with $j \ge i$. Clearly,
$\mathcal{L}_i \cong \Omega^{Z_i}$ in a natural sense. Note that
$\mathcal{L}_i \subset \mathcal{L}_{i+1}$ and that $\bigcup_i
\mathcal{L}_i$ is dense in $\Omega^{\Z^2}$. In each $\mathcal{L}_i$ we
apply the method of \cite{Tr2} to construct a $G_\delta$ dense set
$\mathcal{R}_i$ of recurrent Lorentz gases. Let us sketch this method.

In what follows, whenever we mention circles, balls, annuli, we will
always mean circles, balls, annuli in $\Z^2$, relative to the norm $\|
\cdot \|$ and centered in the origin. Fix $i$ and denote by $\xi$ a
\con\ of cells in a ball of $Z_i$, equivalently, a vector of
$\Omega^B$, where $B$ is the restriction to $Z_i$ of a ball in $\Z^2$;
$B$ will be referred to as the \emph{support} of $\xi$.  Clearly,
there are countably many such (finite) \con s, so we can index them as
$( \xi_k )_{k \in \Z^+}$ (the dependence on $i$ is suppressed). For
each such $k$, let us construct a finite \con\ $\eta_k$ such that:
\begin{itemize}
\item the support of $\eta_k$ is a ball of radius $\rho > \rho_1$,
  where $\rho_1$ is radius of the support of $\xi_k$, and the
  restriction of $\eta_k$ to the support of $\xi_k$ is $\xi_k$;
\item there exists a positive integer $\rho_2 \in (\rho_1, \rho)$ such
  that, if $\rho_1 < \|\bn\| \le \rho_2$, $\l_\bn = 1$, i.e., there is
  a ``blocking annulus'' (of type $\cs^1$) of radii $\rho_1, \rho_2$;
  $\rho_2$ must be so large that the phase space \me\ of all the line
  elements based in the inner circle of the annulus, whose \tr ies
  reach the outer circle before coming back to inner circle, does not
  exceed $1/k$;
\item $\rho - \rho_2 \ge k$ and, for $\bn \not\in Z_i$ with $\rho_2 <
  \|\bn\| \le \rho$, $\l_\bn = m$, i.e., the outer part of the \con\
  $\eta_k$ (except for the blocking circles $D_j$, which do not belong
  to $Z_i$) is a non-blocking annulus (of type $\cs^m$) with thickness
  $\ge k$.
\end{itemize}
Then, let us denote by $\mathcal{A}_i^k$ the cylinder in
$\mathcal{L}_i$ defined by all the \con s in $Z_i$ that coincide with
$\eta_k$ on its support. It is not hard to show that $\mathcal{A}_i^k$
is open in $\mathcal{L}_i$ w.r.t.\ the metric (\ref{dist}). Hence
\begin{equation}
  \label{ri}
  \mathcal{R}_i := \bigcap_{n \in \Z^+} \bigcup_{k \ge n} \mathcal{A}_i^k
\end{equation}
is a $G_\delta$ set that is clearly dense in $\mathcal{L}_i$. The
recurrence of any LG $\l \in \mathcal{R}_i$ is proved essentially as
in Section \ref{subs-pf1}, by showing that, if $W$ is a wandering set,
then $\mu(W \cap \ps_\bn) \le 1/k$, for all $\bn \in \Z^2$ and $k$
large enough. Also, the construction of the non-blocking annulus in
each $\eta_k$ and the hypothesis on the non-blocking horizontal
constant factors (cf.~Theorem \ref{thm2}) imply that $\l$ has infinite
horizon.

The presence of the blocking circles $D_j$ is necessary to ensure that
the set of all the \tr ies with initial positions, say, in $C_\bn$,
stay confined, within time $t\in \N$, to a portion of the LG that
comprises at most $C t^\alpha$ cells, for some $C, \alpha > 0$. This
makes the equivalent of Lemma \ref{lm-tech} hold, which in turn yields
\hyp ity, \erg ity and the other statements of Theorem \ref{thm2}.

Finally, $\mathcal{R} := \bigcup_{i} \mathcal{R}_i$ is a dense
uncountable set of LGs in $\Omega^{\Z^2}$ that possess all the sought
properties.

\end{document}